\documentclass[a4paper,10pt,twoside]{article}

\usepackage{amsmath, amssymb, amsthm, ascmac}
\usepackage{url}
\usepackage[dvipdfmx]{graphicx}

\usepackage{mathptmx}
\newtheorem{thm}{Theorem}[section]
\newtheorem{prop}{Proposition}[section]
\newtheorem{lem}{Lemma}[section]
\newtheorem{cor}{Corollary}[section]
\newtheorem{remark}{Remark}[section]
\newtheorem{example}{Example}[section]

\numberwithin{equation}{section}
\allowdisplaybreaks

\def\R{\mathbb{R}}

\def\toP{\stackrel{p}{\longrightarrow}}
\def\toD{\stackrel{{\cal L}}{\longrightarrow}}

\def\E{\mathbb{E}}

\def\Var{\mathrm{Var}}
\def\Cov{\mathrm{Cov}}

\newcommand{\bi}{\begin{itemize}}
\newcommand{\ei}{\end{itemize}}

\title{M-estimation for Gaussian processes with time-inhomogeneous drifts from high-frequency data}
\author{Yasutaka SHIMIZU\\ {\it Department of Applied Mathematics, Waseda University}}
\date{November 9, 2026}

\begin{document}
\maketitle

\begin{abstract}
We study drift and covariance estimation for a stationary Gaussian process with an integrated deterministic drift, observed at high frequency over an expanding time horizon. We establish consistency and asymptotic normality of an increment least-squares estimator under conditions on drift accumulation and the full covariance of the weighted increments. For directly Riemann integrable drifts, this estimator need not be consistent, as an explicit example shows. Covariance parameters can nevertheless be estimated consistently when the drift decays sufficiently fast. For smooth short-range kernels, we prove a joint central limit theorem for the centered sample variance and the normalized increment second moment, with rate $\sqrt{T_n}$. These statistics require no drift estimate and identify both parameters of the Gaussian kernel. The results clarify how drift accumulation and increment dependence affect the two estimation problems.
\end{abstract}


\section{Introduction}
Gaussian processes provide a flexible framework for modeling temporal dependence;
see Ibragimov and Rozanov~\cite{ir78} and Rasmussen and Williams~\cite{rw06}.
We consider a stationary Gaussian process with a time-dependent deterministic drift,
\begin{equation}\label{model}
X_t=Z_t+m_\xi(t),\qquad m_\xi(t)=\int_0^t\mu_\xi(s)\,ds,
\end{equation}
where $Z$ is centered with covariance $\E[Z_tZ_s]=K_\sigma(t-s)$.
The kernel is defined on $\R$ and is even. Our aim is to estimate the drift
parameter $\xi$ and the covariance parameter $\sigma$ from observations at
$t_i=ih_n$, $i=0,\ldots,n$, under
\begin{equation}\label{sampling}
h_n\to0,\qquad T_n=nh_n\to\infty.
\end{equation}

A Gaussian contrast based on individual increments avoids inversion of the full
covariance matrix. Its simplicity is useful, but the dependence between increments
must be accounted for when deriving convergence rates and limiting variances.
This issue also arises in composite likelihood estimation; see
Varin et al.~\cite{vetal11}.

For drift estimation, we use the increment least-squares contrast studied in
Shimizu~\cite{shimizu26}. Its asymptotic behavior depends on the accumulation of
the deterministic drift and the full covariance of the weighted increments.
We give conditions for consistency and asymptotic normality, together with their
proofs. We also show that a directly Riemann integrable drift can lead to a
nonvanishing estimation error, even as the sampling interval tends to zero and
the observation horizon increases.

Our main additional result concerns covariance estimation in this setting.
If the drift decays sufficiently fast, $m_{\xi_0}(t)$ approaches a constant.
The centered sample variance and the normalized increment second moment then
estimate $K(0)$ and $-K''(0)$ without fitting the drift. For smooth short-range
kernels, we establish their joint asymptotic normality at rate $\sqrt{T_n}$.
The limiting covariance depends on integrals of the kernel and its derivatives.
For the Gaussian kernel, these two moments identify both covariance parameters,
even in an example where the increment drift estimator is inconsistent.

Section~\ref{sec:drift} presents the drift estimation results, and
Section~\ref{sec:localized} gives the inconsistent DRI example.
Section~\ref{sec:cov} develops covariance estimation.
Proofs are collected in Section~\ref{sec:proofs}, followed by concluding remarks.

\paragraph{Notation.}
Write $\Delta_i^nY=Y_{t_i}-Y_{t_{i-1}}$ and use Euclidean and Frobenius norms
for vectors and matrices. The symbols $\toP$, $\toD$, and $\to_{L^2}$ denote
convergence in probability, distribution, and mean square, respectively.
For a vector $v$, $v^{\otimes2}=vv^\top$. The constants denoted by $C$ may change
from line to line. A matrix inequality $G>0$ means positive definiteness.
The covariance kernel without a subscript always denotes the true kernel
$K=K_{\sigma_0}$.

\section{Drift estimation by local increments}\label{sec:drift}
Let $\Xi\subset\R^p$ be bounded, open, and convex, and let
$\xi_0\in\Xi$. Optimization is over its compact closure. The true drift is
$\mu_{\xi_0}$. Define
\begin{equation}\label{contrast}
H_n(\xi)=\sum_{i=1}^n\{\Delta_i^nX-h_n\mu_\xi(t_{i-1})\}^2,
\qquad \widehat\xi_n\in\arg\min_{\xi\in\overline\Xi}H_n(\xi),
\end{equation}
with a measurable choice of minimizer. The local Gaussian contrast is
\begin{equation}\label{lg-contrast}
\ell_n(\xi,\sigma)=\frac{H_n(\xi)}{nq_n(\sigma)}
 +\log\{q_n(\sigma)/h_n^2\},\qquad
q_n(\sigma)=2\{K_\sigma(0)-K_\sigma(h_n)\}>0.
\end{equation}
For fixed $\sigma$, minimizing \eqref{lg-contrast} over $\xi$ is equivalent to
\eqref{contrast}. This remains true for joint minimization over a Cartesian
parameter space, provided a minimizer exists. We therefore study the drift
estimator separately.

\subsection{Assumptions for consistency}
We impose the following conditions.
\begin{description}
\item[(D1)] For some $\alpha\in(0,2]$,
\begin{equation}\label{rowbound}
\sup_{1\le i\le n}\sum_{j=1}^n
 |\Cov(\Delta_i^nZ,\Delta_j^nZ)|\le Ch_n^\alpha.
\end{equation}
\item[(D2)] For every $t\ge0$, $\mu_\xi(t)$ is continuously differentiable in
$\xi$ on a neighborhood of $\overline\Xi$. It is continuously differentiable in
$t$ and $\sup_{t,\xi}|\partial_t\mu_\xi(t)|<\infty$.
\item[(D3)] There exist $a_n\to\infty$ and a continuous nonnegative function
$Q$ on $\overline\Xi$ such that, uniformly in $\xi\in\overline\Xi$,
\begin{equation}\label{separation}
\frac1{a_n}\sum_i\{\mu_\xi(t_{i-1})-\mu_{\xi_0}(t_{i-1})\}^2
 \to Q(\xi),\qquad Q(\xi)=0\ \Longleftrightarrow\ \xi=\xi_0.
\end{equation}
\item[(D4)] There is a deterministic envelope $M$ such that
\[
\sup_{\xi\in\overline\Xi}|\partial_\xi\mu_\xi(t)|\le M(t),
\qquad \sum_iM(t_{i-1})^2=O(a_n).
\]
\end{description}
The sequence $a_n$ measures drift separation. Condition (D4) bounds parameter
sensitivity on the same scale.

\begin{thm}[Drift consistency]\label{thm:consist}
Suppose \eqref{sampling} and (D1)--(D4) hold, and
\begin{equation}\label{consistency-balances}
a_nh_n^{2-\alpha}\to\infty,\qquad nh_n^2/a_n\to0.
\end{equation}
Then $\widehat\xi_n\toP\xi_0$.
\end{thm}

\begin{remark}[Smooth and nonsmooth kernels]\label{rem:row}
If $K$ is even and twice continuously differentiable, $K'(0)=0$, and
$K''\in L^1(\R)$, then (D1) holds with $\alpha=1$. Indeed, for $k\ge1$,
\[
|2K(kh)-K((k-1)h)-K((k+1)h)|
\le h\int_{(k-1)h}^{(k+1)h}|K''(u)|\,du.
\]
Summing this bound, and using $2\{K(0)-K(h)\}=O(h^2)$ for the diagonal,
gives \eqref{rowbound}. Thus single-increment variance of order $h^2$ does not
imply a covariance row sum of order $h^2$. For the OU kernel
$K(t)=a e^{-b|t|}$, a geometric-series calculation also gives $\alpha=1$,
although single-increment variance is of order $h$.
\end{remark}

\subsection{Asymptotic normality}
Write $w_{i,n}=\partial_\xi\mu_{\xi_0}(t_{i-1})$ and
\begin{equation}\label{fullscore}
S_n=\sum_iw_{i,n}\Delta_i^nZ,\qquad
V_n=\Var(S_n)=\sum_{i,j}w_{i,n}w_{j,n}^\top
                 \Cov(\Delta_i^nZ,\Delta_j^nZ).
\end{equation}
In addition to (D1)--(D4), assume the following.
\begin{description}
\item[(D5)] The family is twice continuously differentiable in $\xi$ on a
neighborhood of $\overline\Xi$, and
\[
I_n=\frac1{a_n}\sum_iw_{i,n}w_{i,n}^\top\to I>0.
\]
\item[(D6)] On a bounded convex neighborhood $U$ of $\xi_0$ with
$\overline U\subset\Xi$, there are envelopes $N,L$ and $\rho\in(0,1]$ such that
\begin{align*}
\sup_{\xi\in U}|\partial_\xi^2\mu_\xi(t)|&\le N(t),\\
|\partial_\xi^2\mu_\xi(t)-\partial_\xi^2\mu_\eta(t)|
 &\le L(t)|\xi-\eta|^\rho,\qquad \xi,\eta\in U,\\
\sum_i\{N(t_{i-1})^2+L(t_{i-1})^2\}&=O(a_n).
\end{align*}
\end{description}
For families linear in $\xi$, (D6) holds with $N=L=0$.

\begin{thm}[Asymptotic normality]\label{thm:asy-norm}
Suppose the conditions of Theorem~\ref{thm:consist}, (D5), and (D6) hold.
Suppose that for a deterministic $v_n>0$,
\begin{equation}\label{score-assumptions}
V_n/v_n\to\Gamma>0,\qquad
R_n=\frac{a_nh_n}{\sqrt{v_n}}\to\infty,\qquad
\frac{na_nh_n^4}{v_n}\to0.
\end{equation}
Then
\[
R_n(\widehat\xi_n-\xi_0)\toD N(0,I^{-1}\Gamma I^{-1}).
\]
\end{thm}

\begin{remark}
Since $S_n$ is Gaussian, its limiting distribution is determined by $V_n$.
The scale $v_n$ need not diverge. Taking $v_n=a_nh_n^\beta$ gives the
normalization $\sqrt{a_nh_n^{2-\beta}}$, provided the full covariance limit
in \eqref{score-assumptions} holds. The expansion of $K(0)-K(h)$ alone
does not establish this limit.
\end{remark}

\begin{example}[Periodic drift]\label{ex:periodic}
Let $\mu_\xi(t)=\xi_1\sin(\omega t)+\xi_2\cos(\omega t)$ with fixed $\omega>0$.
For either the Gaussian kernel $a e^{-bt^2/2}$ or the OU kernel $a e^{-b|t|}$,
take $a_n=n$. Then $Q(\xi)=|\xi-\xi_0|^2/2$, $I=I_2/2$, and (D1)--(D6)
hold with $\alpha=1$. Under $T_nh_n\to0$,
\[
\frac{V_n}{T_n}\to\omega^2 C_K(\omega)I_2,
\qquad C_K(\omega)=\int_0^\infty K(u)\cos(\omega u)\,du>0.
\]
Consequently
\[
\sqrt{T_n}(\widehat\xi_n-\xi_0)
 \toD N(0,4\omega^2 C_K(\omega)I_2).
\]
Here $C_K(\omega)=a\sqrt{\pi/(2b)}e^{-\omega^2/(2b)}$ for the Gaussian kernel
and $ab/(b^2+\omega^2)$ for the OU kernel. A proof of the score limit is given
in Section~\ref{proof:periodic}.
\end{example}

\section{Localized drifts}\label{sec:localized}
A nonnegative function $f$ on $[0,\infty)$ is directly Riemann integrable (DRI)
if its upper and lower uniform-mesh Riemann sums on the half-line converge to the
same finite integral as the mesh tends to zero. For a signed function, require this
of its positive and negative parts. See Feller~\cite{f71}.
For a nonzero DRI squared drift difference, the natural accumulation is
$a_n\asymp h_n^{-1}$. Under a smooth short-range kernel with $\alpha=1$,
$a_nh_n$ does not diverge. The next example shows that the estimator can indeed be inconsistent.

\begin{prop}[Exponential drift]\label{prop:dri}
Let $\mu_\xi(t)=\xi e^{-t}$, $t\ge0$, and $K(t)=a e^{-bt^2/2}$, $a,b>0$.
Under \eqref{sampling}, the unconstrained minimizer of \eqref{contrast} satisfies
\begin{equation}\label{dri-limit}
\widetilde\xi_n-\xi_0\to_{L^2}2S,\qquad
S=-Z_0+\int_0^\infty e^{-t}Z_t\,dt,
\end{equation}
where
\begin{equation}\label{dri-var}
\Var(S)=K(0)-\int_0^\infty e^{-t}K(t)\,dt>0.
\end{equation}
If optimization is restricted to a compact interval $[l,u]$ with
$l<\xi_0<u$, its limit is the projection of $\xi_0+2S$ onto $[l,u]$,
which is not almost surely $\xi_0$. Neither estimator is consistent, and the
normalization $h_n^{-1/2}(\widehat\xi_n-\xi_0)$ is not tight.
\end{prop}

\begin{remark}
The diagonal terms in
$\Var(\sum_i e^{-t_{i-1}}\Delta_i^nZ)$ sum to $O(h_n)$, whereas the full
variance converges to the positive value \eqref{dri-var}. The off-diagonal
terms therefore determine the limiting error. This inconsistency result is
specific to the increment estimator.
\end{remark}

\section{Covariance estimation}\label{sec:cov}
We now estimate covariance parameters without fitting the drift. The results
apply to Proposition~\ref{prop:dri}, under the following smoothness and
localization conditions.

\subsection{Assumptions and estimators}
\begin{description}
\item[(C1)] The even covariance $K$ is four times continuously differentiable on
$\R$, with $k=K(0)>0$ and $c=-K''(0)>0$. Also
$K,K',K''\in L^1(\R)$. We use a continuous modification of $Z$ and its
mean-square derivative $\dot Z$. The latter may be taken continuous under these
assumptions, and $Z_t-Z_s=\int_s^t\dot Z_u\,du$ in mean square.
\item[(C2)] The true deterministic drift $\mu=\mu_{\xi_0}$ is continuously
differentiable and, for some $\eta>0$,
\begin{equation}\label{localization}
|\mu(t)|+|\mu'(t)|\le C(1+t)^{-2-\eta},\qquad t\ge0.
\end{equation}
\end{description}
Condition (C2) is stronger than DRI. The Gaussian kernel satisfies (C1), and
the exponential drift in Proposition~\ref{prop:dri} satisfies (C2).

Define, using the $n$ left-endpoint levels and the $n$ increments,
\begin{equation}\label{cov-est}
\overline X_n=\frac1n\sum_{i=1}^nX_{t_{i-1}},\qquad
\widehat k_n=\frac1n\sum_i(X_{t_{i-1}}-\overline X_n)^2,
\qquad
\widehat c_n=\frac1{nh_n^2}\sum_i(\Delta_i^nX)^2.
\end{equation}
Centering removes the limiting level
$M_\infty=\int_0^\infty\mu(s)\,ds$. Neither statistic requires an estimate
of $\xi_0$.

\begin{thm}[Consistency of covariance moments]\label{thm:M-consist}
Under (C1), (C2), and \eqref{sampling},
$(\widehat k_n,\widehat c_n)\toP(k,c)$.
This conclusion does not require consistency of any drift estimator.
\end{thm}

\begin{thm}[Joint limit of covariance moments]\label{thm:M-asy-norm}
Under (C1), (C2), \eqref{sampling}, and
\begin{equation}\label{cov-mesh}
\sqrt{T_n}\,h_n\to0,
\end{equation}
we have
\begin{equation}\label{cov-clt}
\sqrt{T_n}\begin{pmatrix}\widehat k_n-k\\\widehat c_n-c\end{pmatrix}
 \toD N(0,\Omega),\qquad
\Omega=4\int_0^\infty
\begin{pmatrix}
K(u)^2 & K'(u)^2\\ K'(u)^2 & K''(u)^2
\end{pmatrix}\,du.
\end{equation}
The matrix $\Omega$ may be singular.
\end{thm}

\begin{remark}[Local Gaussian covariance fitting]
The smooth-kernel local contrast has the leading form
$\widehat c_n/c+\log c$ when localized drift is ignored. Its unrestricted
minimizer in $c>0$ is $\widehat c_n$. This identifies a curvature parameter,
not an arbitrary vector $\sigma$. For Gaussian kernels, $c=ab$ alone cannot
separate $a$ and $b$. The additional centered level moment estimates $k=a$.
The preceding theorems concern the explicit moment estimators
\eqref{cov-est}.
\end{remark}

\begin{prop}[Curvature estimation with a fitted drift]\label{prop:fit-curvature}
Suppose (C1), (C2), and \eqref{sampling}. Let $\widehat\xi_n$ be any random
parameter taking values in a family satisfying
\[
\sup_\xi|\mu_\xi(t)|\le C(1+t)^{-2-\eta}.
\]
For $\widehat c_n^{\rm fit}=(nh_n^2)^{-1}\sum_i
\{\Delta_i^nX-h_n\mu_{\widehat\xi_n}(t_{i-1})\}^2$,
\[
\widehat c_n^{\rm fit}-\widehat c_n=O_P(T_n^{-1}).
\]
Thus both the consistency conclusion and the limit in
Theorem~\ref{thm:M-asy-norm} remain valid with $\widehat c_n^{\rm fit}$.
No convergence of $\widehat\xi_n$ is required here.
\end{prop}

\subsection{Recovering covariance parameters}
Suppose the covariance family is parameterized by $\sigma\in\Pi\subset\R^q$.
Put $F(\sigma)=(K_\sigma(0),-K_\sigma''(0))^\top$. A sufficient regular
identification condition is the existence of a continuously differentiable map
$g$ on a neighborhood of $F(\sigma_0)$ such that
$g(F(\sigma))=\sigma$ for $\sigma$ near $\sigma_0$. Define
$\widehat\sigma_n=g(\widehat k_n,\widehat c_n)$ on that neighborhood and
arbitrarily outside it. Then the preceding theorems and the delta method give
\begin{equation}\label{parameter-clt}
\widehat\sigma_n\toP\sigma_0,\qquad
\sqrt{T_n}(\widehat\sigma_n-\sigma_0)
 \toD N(0,Dg\,\Omega\,Dg^\top).
\end{equation}
The second conclusion requires \eqref{cov-mesh}. The existence of such a local
left inverse forces $q\le2$. More parameters require additional identifiable
moments, for example at fixed positive lags, with their own limit theory.
Single-time Gaussian moments alone depend only on $K(0)$ and cannot identify
several kernel parameters.

\begin{cor}[Gaussian kernel]\label{cor:gaussian}
Let $K_{a,b}(t)=ae^{-bt^2/2}$ with $a,b>0$, and suppose (C2).
Set
\[
\widehat a_n=\widehat k_n,\qquad
\widehat b_n=\widehat c_n/\widehat k_n,
\]
with an arbitrary positive definition on the event $\widehat k_n=0$.
Under \eqref{sampling} these are consistent. Under \eqref{cov-mesh} as well,
\begin{equation}\label{gauss-clt}
\sqrt{T_n}\begin{pmatrix}\widehat a_n-a\\\widehat b_n-b\end{pmatrix}
 \toD N\left(0,\sqrt\pi
\begin{pmatrix}
2a^2/\sqrt b & -a\sqrt b\\
-a\sqrt b & 3b^{3/2}/2
\end{pmatrix}\right).
\end{equation}
The matrix is positive definite. In particular, for the exponential drift in
Proposition~\ref{prop:dri}, these covariance estimators are consistent although
the increment drift estimator is not.
\end{cor}

\subsection{Estimation from residuals}
For a fitted drift let $Y_t=X_t-m_{\widehat\xi_n}(t)$.
With $\mu_\xi(t)=\xi e^{-t}$ this is
\[
Y_t=Z_t+(\xi_0-\widehat\xi_n)(1-e^{-t}).
\]
For the unconstrained estimator in Proposition~\ref{prop:dri},
$\widehat\xi_n-\xi_0\to_{L^2}2S$. Consequently,
\begin{equation}\label{bad-residual}
\frac1n\sum_iY_{t_{i-1}}^2\toP k+4S^2.
\end{equation}
The analogous statement for a constrained fit uses its projected random error.
Thus an uncentered residual second moment is not in general consistent for $k$.
Centering the observations as in \eqref{cov-est} avoids this random level error.

For nonlocalized drifts, detrending requires additional control of the residual
error. The following proposition gives sufficient conditions.
\begin{prop}[Detrending errors]\label{prop:transfer}
Suppose (C1), \eqref{sampling}, and \eqref{cov-mesh}. Let observed residuals
at the grid satisfy $Y_{t_i}=Z_{t_i}+r_{i,n}$, $i=0,\ldots,n$, and put
\[
\bar r_n=\frac1n\sum_{i=1}^nr_{i-1,n},\quad
A_n=\frac1n\sum_i(r_{i-1,n}-\bar r_n)^2,\quad
B_n=\frac1n\sum_i\left(\frac{r_{i,n}-r_{i-1,n}}{h_n}\right)^2.
\]
If $\sqrt{T_n}(\sqrt{A_n}+\sqrt{B_n})\toP0$, the centered level and
increment estimators formed from $Y$ satisfy \eqref{cov-clt}.
For consistency alone, $A_n+B_n\toP0$ suffices, without \eqref{cov-mesh}.
For localized drifts, Theorems~\ref{thm:M-consist}--\ref{thm:M-asy-norm}
use sharper bounds.
\end{prop}
In particular, pointwise consistency of $\widehat\xi_n$ alone does not control
integrated detrending errors over an expanding time horizon.

\section{Proofs}\label{sec:proofs}
\subsection{Preliminary bounds}
\begin{lem}[Weighted increments]\label{lem:m-DZ}
Under (D1), for deterministic scalar weights $u_i$,
\[
\Var\left(\sum_i u_i\Delta_i^nZ\right)
 \le Ch_n^\alpha\sum_i u_i^2.
\]
In particular, $\sum_i u_i^2=O(a_n)$ gives
$\sum_i u_i\Delta_i^nZ=O_P(\sqrt{a_nh_n^\alpha})$.
\end{lem}
\begin{proof}
Use $2|u_iu_j|\le u_i^2+u_j^2$ in the full covariance double sum and apply
\eqref{rowbound}. The bound holds componentwise for vector or matrix weights.
\end{proof}

\begin{lem}[Gaussian supremum bound]\label{lem:chain}
Let $G$ be a continuous centered Gaussian field on a fixed compact subset
$E\subset\R^p$. Suppose, for some $s>0$ and $\rho>0$,
\[
\|G(x)-G(y)\|_2\le Cs|x-y|^\rho,\qquad \|G(x_*)\|_2\le Cs
\]
for a fixed $x_*\in E$. Then $\E\sup_{x\in E}|G(x)|\le C's$, with $C'$
independent of $s$.
\end{lem}
\begin{proof}
Choose Euclidean nets of mesh $2^{-j}$ with at most $C2^{jp}$ points and chain
each point through its nearest net points. The Gaussian increments at level $j$
have standard deviations at most $Cs2^{-j\rho}$. The expected maximum of the
absolute values of $m$ centered Gaussian variables with standard deviations at
most $v$ is at most $v\sqrt{2\log(2m)}$. Hence the expected supremum of the
chain increments is bounded by
$Cs\sum_{j\ge0}2^{-j\rho}\sqrt{1+j}<\infty$.
Add the anchor and use continuity to pass from the union of the nets to $E$.
\end{proof}

\subsection{Proof of Theorem~\ref{thm:consist}}
Put $d_i(\xi)=\mu_\xi(t_{i-1})-\mu_{\xi_0}(t_{i-1})$ and
\[
r_i^n=\int_{t_{i-1}}^{t_i}\{\mu_{\xi_0}(s)-\mu_{\xi_0}(t_{i-1})\}\,ds.
\]
By (D2), $|r_i^n|\le Ch_n^2$. Exact expansion gives
\begin{equation}\label{contrast-expansion}
\frac{H_n(\xi)-H_n(\xi_0)}{a_nh_n^2}
=\frac1{a_n}\sum_i d_i(\xi)^2
 -\frac2{a_nh_n}\sum_i d_i(\xi)\Delta_i^nZ
 -\frac2{a_nh_n}\sum_i d_i(\xi)r_i^n.
\end{equation}
The first term converges uniformly to $Q$ by (D3). For the second, (D4) and
Lemma~\ref{lem:m-DZ} give the canonical-metric bound
\[
\left\|\frac2{a_nh_n}\sum_i\{d_i(\xi)-d_i(\eta)\}\Delta_i^nZ\right\|_2
\le \frac{C|\xi-\eta|}{\sqrt{a_nh_n^{2-\alpha}}}.
\]
The field vanishes at $\xi_0$, so Lemma~\ref{lem:chain} and
\eqref{consistency-balances} give uniform convergence to zero in probability.
For the last term, Cauchy--Schwarz and (D3) give the uniform bound
\[
\frac{C h_n}{a_n}\sum_i|d_i(\xi)|
\le C\sqrt{\frac{nh_n^2}{a_n}}
 \left\{\frac1{a_n}\sum_i d_i(\xi)^2\right\}^{1/2}=o(1).
\]
On the compact parameter set, continuity and unique minimization of $Q$ imply
strict positive separation outside any neighborhood of $\xi_0$. Since the
normalized contrast is zero at $\xi_0$, its minimizer is consistent.
\qed

\subsection{Proof of Theorem~\ref{thm:asy-norm}}
Let $\Psi_n=\partial_\xi H_n$. On the event that $\widehat\xi_n\in U$, which has
probability tending to one, Taylor's formula yields
\[
0=\Psi_n(\xi_0)+2a_nh_n^2J_n(\widehat\xi_n-\xi_0),\qquad
J_n=\frac1{2a_nh_n^2}\int_0^1
\partial_\xi^2H_n(\xi_0+u(\widehat\xi_n-\xi_0))\,du.
\]
We first show $J_n\toP I$. The leading term is the averaged outer product of
$\partial_\xi\mu_\xi(t_{i-1})$ divided by $a_n$. Its difference from $I_n$
is bounded by
\[
C|\widehat\xi_n-\xi_0|\frac1{a_n}\sum_i M(t_{i-1})N(t_{i-1})=o_P(1).
\]
The residual part of the Hessian splits into noise, discretization, and
parameter-shift terms. For each entry of the noise field
\[
G_n(\xi)=\frac1{a_nh_n}\sum_i\partial_\xi^2\mu_\xi(t_{i-1})\Delta_i^nZ,
\quad \xi\in\overline U,
\]
(D1), (D6), and Lemmas~\ref{lem:m-DZ}--\ref{lem:chain} imply
\[
\E\sup_{\xi\in\overline U}|G_n(\xi)|
\le C(a_nh_n^{2-\alpha})^{-1/2}\to0.
\]
Use continuity on $\overline U$, shrinking $U$ if necessary. The other two
terms are bounded, respectively, by
\[
\frac{Ch_n}{a_n}\sum_iN(t_{i-1})
\le C\sqrt{nh_n^2/a_n},\qquad
\frac{C|\widehat\xi_n-\xi_0|}{a_n}\sum_iM(t_{i-1})N(t_{i-1}),
\]
and tend to zero. This proves the Hessian limit.
Since $\Psi_n(\xi_0)=-2h_n\sum_iw_{i,n}(\Delta_i^nZ+r_i^n)$,
\[
R_n(\widehat\xi_n-\xi_0)
=J_n^{-1}\left\{\frac{S_n}{\sqrt{v_n}}
              +\frac1{\sqrt{v_n}}\sum_iw_{i,n}r_i^n\right\}.
\]
The second term has norm at most
$C\sqrt{na_nh_n^4/v_n}=o(1)$.
The first is exactly centered Gaussian with covariance $V_n/v_n\to\Gamma$.
Slutsky's theorem proves the result on an event of probability tending to one,
which suffices for the unconditional limit. \qed

\subsection{The periodic example}\label{proof:periodic}
Let $w(t)=(\sin(\omega t),\cos(\omega t))^\top$. Define by integration by parts
\[
Y_T=w(T)Z_T-w(0)Z_0-\int_0^T w'(t)Z_t\,dt.
\]
An exact discrete integration by parts gives
\[
S_n-Y_{T_n}=\sum_i\int_{t_{i-1}}^{t_i}w'(t)(Z_t-Z_{t_i})\,dt.
\]
For the two kernels, $\|Z_{t+h}-Z_t\|_2\le C h^{\beta/2}$ for small $h$,
with $\beta=2$ or $1$, respectively. Under $T_nh_n\to0$,
\[
\|S_n-Y_{T_n}\|_2\le CT_nh_n^{\beta/2}=o(\sqrt{T_n}).
\]
Integrability of $K$ and trigonometric product identities give
\[
\frac1T\Var\left(\int_0^T\cos(\omega t)Z_t\,dt\right)\to C_K(\omega).
\]
The sine term has the same limit and the normalized cross covariance tends to
zero. The nonoscillatory contribution to the variance is
\[
\int_0^T(T-u)K(u)\cos(\omega u)\,du.
\]
The oscillatory remainder is bounded by a constant times $\|K\|_{L^1}$.
Boundary contributions vanish after normalization.
It follows that $V_n/T_n\to\omega^2C_K(\omega)I_2$.
All the remaining balances of Theorem~\ref{thm:asy-norm} hold:
$a_nh_n=T_n\to\infty$, $nh_n^2/a_n=h_n^2\to0$, and
$na_nh_n^4/T_n=T_nh_n^2\to0$. The stated integrals for $C_K$ follow by
Gaussian and exponential integration. \qed

\subsection{Proof of Proposition~\ref{prop:dri}}
Set $g_i=e^{-t_{i-1}}$ and $d_h=(1-e^{-h})/h$. Since the true drift increment
is exactly $\xi_0e^{-t_{i-1}}(1-e^{-h_n})$,
\[
\widetilde\xi_n=\xi_0d_{h_n}
 +\frac{\sum_i g_i\Delta_i^nZ}{h_n\sum_i g_i^2}.
\]
Here $d_{h_n}\to1$ and
$h_n\sum_i g_i^2=h_n(1-e^{-2T_n})/(1-e^{-2h_n})\to1/2$.
Discrete summation by parts gives
\[
\sum_i g_i\Delta_i^nZ
=e^{-(T_n-h_n)}Z_{T_n}-Z_0
 +(1-e^{-h_n})\sum_{i=1}^{n-1}e^{-(i-1)h_n}Z_{t_i}.
\]
Mean-square continuity gives Riemann convergence on each bounded interval;
the exponentially weighted tails are uniformly small in $L^2$ by
$\|Z_t\|_2=\sqrt{k}$. Therefore this expression converges in $L^2$ to $S$.
Writing $J=\int_0^\infty e^{-t}K(t)\,dt$, direct integration gives
\[
\Cov\left(Z_0,\int_0^\infty e^{-t}Z_t\,dt\right)=J,\qquad
\Var\left(\int_0^\infty e^{-t}Z_t\,dt\right)=J.
\]
Consequently $\Var(S)=k-J$. For the Gaussian kernel, $K(t)<k$ for all $t>0$,
so $J<k$. Projection onto an interval is Lipschitz and yields the stated
constrained limit. Since $\xi_0$ is interior and $S$ is nondegenerate, the limit
error is not almost surely zero. Tightness after multiplication by a diverging
sequence would imply consistency, which is impossible. \qed

\subsection{A continuous-time quadratic Gaussian limit}\label{sec:quadratic}
Under (C1), put $W_t=\dot Z_t$. Gaussian differentiation gives
\[
\E[W_0W_u]=-K''(u),\qquad
\E[Z_0W_u]=K'(u),\qquad \E[W_0Z_u]=-K'(u).
\]
Mean-square increments of both $Z$ and $W$ have norm at most $C|t-s|$
for $|t-s|\le1$. For $W$, this follows by expanding $K''$ at zero and using
evenness and $K^{(4)}(0)<\infty$. Gaussian fourth-moment bounds then provide
continuous versions. Integration of the mean-square derivative agrees with
$Z_t-Z_s$, as follows by approximating the integral in mean square.

\begin{lem}\label{lem:quadratic}
Under (C1),
\[
\frac1{\sqrt T}\begin{pmatrix}
\int_0^T(Z_t^2-k)\,dt\\ \int_0^T(W_t^2-c)\,dt
\end{pmatrix}\toD N(0,\Omega),
\]
where $\Omega$ is given in \eqref{cov-clt}. Also
$T^{-1}\int_0^TZ_t\,dt=O_P(T^{-1/2})$.
\end{lem}
\begin{proof}
For jointly centered Gaussian $A,B$, Wick's identity gives
$\Cov(A^2,B^2)=2\Cov(A,B)^2$. The covariance of the displayed vector is therefore
\[
4\int_0^T(1-u/T)
\begin{pmatrix}K(u)^2&K'(u)^2\\K'(u)^2&K''(u)^2\end{pmatrix}\,du\to\Omega.
\]
All these integrals are finite: the covariance entries are bounded and integrable
by (C1). The assertion for the mean follows from
$\Var(\int_0^TZ_t\,dt)\le2T\int_0^\infty|K(u)|\,du$.

For completeness, normality follows from a quadratic-form argument. Regard
$H=(Z,W)$ on $[0,T]$ as a Gaussian random element of $L^2([0,T];\R^2)$,
with covariance operator $C_T$. Its matrix covariance kernel has integrable
entries, so convolution bounds imply $\sup_T\|C_T\|_{\rm op}<\infty$.
For fixed real $u,v$, let $B=\operatorname{diag}(u,v)$ acting pointwise.
The centered quadratic form
\[
Q_T=\int_0^T\{u(Z_t^2-k)+v(W_t^2-c)\}\,dt
\]
has the distribution $\sum_j\lambda_{j,T}(G_j^2-1)$, where $G_j$ are independent
standard normals and $\lambda_{j,T}$ are the eigenvalues of the trace-class
self-adjoint operator $C_T^{1/2}BC_T^{1/2}$. This representation follows by
finite-dimensional Gaussian diagonalization and $L^2$ approximation.
The eigenvalues are uniformly bounded, and
\[
2\sum_j\lambda_{j,T}^2=\Var(Q_T)=O(T).
\]
For each fixed real $z$, expansion of
the logarithm of the characteristic function of $Q_T/\sqrt T$ has quadratic term
$-z^2\Var(Q_T)/(2T)$, and the remainder is bounded by
\[
C_z T^{-3/2}\sum_j|\lambda_{j,T}|^3=O(T^{-1/2}).
\]
It tends to zero. Thus every linear combination has the asserted Gaussian
limit, including zero-variance combinations, and the Cram\'er--Wold argument
completes the proof.
\end{proof}

\subsection{Discretization and localization bounds}\label{sec:approx}
Let $Z_i=Z_{t_{i-1}}$, $U_i=\Delta_i^nZ/h_n$, and $\bar Z_n=n^{-1}\sum_iZ_i$.
Mean-square differentiability gives
\begin{equation}\label{L1approx}
\E\left|\frac1n\sum_iZ_i^2-\frac1{T_n}\int_0^{T_n}Z_t^2\,dt\right|\le Ch_n,
\quad
\E\left|\frac1n\sum_iU_i^2-\frac1{T_n}\int_0^{T_n}W_t^2\,dt\right|\le Ch_n.
\end{equation}
For the first inequality use
$\E|Z_s^2-Z_t^2|\le\|Z_s-Z_t\|_2\|Z_s+Z_t\|_2\le C|s-t|$.
For the second, $U_i=h_n^{-1}\int_{t_{i-1}}^{t_i}W_s\,ds$, so
$\|U_i-W_t\|_2\le Ch_n$ on that cell and $\sup_i\|U_i\|_2\le C$.
The same averaging argument gives
\begin{equation}\label{meanbound}
\bar Z_n=O_P(T_n^{-1/2}+h_n).
\end{equation}
These estimates suffice for consistency when $h_n\to0$ and for the centered
limit when $\sqrt{T_n}h_n\to0$; in particular
$\sqrt{T_n}\bar Z_n^2\toP0$ under that mesh condition.

Now let $d(t)=m_{\xi_0}(t)-M_\infty=-\int_t^\infty\mu(s)\,ds$.
By (C2), $|d(t)|\le C(1+t)^{-1-\eta}$. Define
$d_i=d(t_{i-1})$ and $b_i=h_n^{-1}\int_{t_{i-1}}^{t_i}\mu(s)\,ds$.
Comparison with these monotone integrable envelopes shows, uniformly for $h_n\le1$,
\begin{equation}\label{local-sums}
h_n\sum_i(|d_i|+d_i^2+|b_i|+b_i^2)\le C.
\end{equation}
Since $\sup_i\E|Z_i|+\sup_i\E|U_i|\le C$, Markov's inequality gives
\[
\frac1n\sum_i d_i Z_i=O_P(T_n^{-1}),\qquad
\frac1n\sum_i b_i U_i=O_P(T_n^{-1}).
\]
Subtracting the sample mean removes $M_\infty$ exactly. With
$\bar d_n=n^{-1}\sum_i d_i=O(T_n^{-1})$, expansion of both squares yields
\begin{equation}\label{local-equivalence}
\widehat k_n-\frac1n\sum_i(Z_i-\bar Z_n)^2=O_P(T_n^{-1}),\qquad
\widehat c_n-\frac1n\sum_iU_i^2=O_P(T_n^{-1}).
\end{equation}
Here $\bar Z_n=O_P(1)$ is already sufficient for the first error bound.

\subsection{Proofs of the covariance results}
By the covariance calculation in Lemma~\ref{lem:quadratic}, the continuous-time
averages of $Z_t^2$ and $W_t^2$ converge in $L^2$ to $k$ and $c$.
Equations~\eqref{L1approx}--\eqref{local-equivalence} and
\eqref{meanbound} prove Theorem~\ref{thm:M-consist}.
Under \eqref{cov-mesh}, the errors in \eqref{L1approx} and
\eqref{local-equivalence}, as well as the sample-mean square, are
$o_P(T_n^{-1/2})$. Lemma~\ref{lem:quadratic} and Slutsky's theorem prove
Theorem~\ref{thm:M-asy-norm}.

For Proposition~\ref{prop:fit-curvature}, set
$e_i=\mu_{\widehat\xi_n}(t_{i-1})$. The difference in question is
$n^{-1}\sum_i\{-2e_i(U_i+b_i)+e_i^2\}$.
The deterministic envelope bounds $h_n\sum_i|e_i|$ and
$h_n\sum_i e_i^2$ uniformly. It also bounds the expected absolute weighted
noise sum, without any independence between $\widehat\xi_n$ and $Z$.
The whole expression is $O_P(T_n^{-1})$.

For Corollary~\ref{cor:gaussian}, elementary Gaussian integrals give
\[
\Omega=\sqrt\pi\begin{pmatrix}
2a^2/\sqrt b & a^2\sqrt b\\
a^2\sqrt b & 3a^2b^{3/2}/2
\end{pmatrix},\qquad
Dg=\begin{pmatrix}1&0\\-b/a&1/a\end{pmatrix}.
\]
The map $g(k,c)=(k,c/k)$ is smooth for $k>0$. Applying it to
\eqref{cov-clt} gives \eqref{gauss-clt}; the determinant of the latter covariance
matrix is $2\pi a^2b>0$.

To verify \eqref{bad-residual}, write $e_n=\xi_0-\widehat\xi_n\toP-2S$.
The average of $(1-e^{-t_{i-1}})^2$ tends to one. The average of
$(1-e^{-t_{i-1}})Z_i$ tends to zero by \eqref{meanbound} and the integrable
exponential envelope. The average of $Z_i^2$ tends to $k$ by
\eqref{L1approx}. Expanding the square proves the claim, without independence
between $e_n$ and the process.

Finally, for Proposition~\ref{prop:transfer}, the centered level moment changes
by at most
\[
2\left\{\frac1n\sum_i(Z_i-\bar Z_n)^2\right\}^{1/2}\sqrt{A_n}+A_n,
\]
and the increment moment changes by at most
$2\{n^{-1}\sum_iU_i^2\}^{1/2}\sqrt{B_n}+B_n$.
The corresponding noise averages are $O_P(1)$, which proves both assertions. \qed

\section{Concluding remarks}
We have studied increment-based drift estimation and moment-based covariance
estimation for Gaussian processes with deterministic drifts. Drift consistency
requires sufficient signal accumulation relative to the full weighted increment
variance. For a localized drift, the increment estimator may instead have a
nondegenerate random limit.

Covariance estimation remains possible under the localization and smoothness
conditions of Section~\ref{sec:cov}. Centering the levels and averaging squared
increments give consistent estimates of the marginal variance and local curvature,
with a joint $\sqrt{T_n}$ limit. The Gaussian-kernel example shows that both
covariance parameters can be recovered without a consistent drift estimate.

Extensions to nonsmooth kernels and more slowly decaying drifts require different
approximation bounds. Joint limits for drift and covariance estimators also depend
on the drift regime. These questions are left for future work.

\section*{Statements and declarations}
\paragraph{Funding}
This work is partially supported by JSPS KAKENHI Grant-in-Aid for Scientific
Research (C), Grant No.~24K06875, and Japan Science and Technology Agency CREST,
Grant No.~JPMJCR2115.
\paragraph{Competing interests}
The author has no relevant financial or non-financial interests to disclose.
\paragraph{Data availability}
Data sharing is not applicable to this article as no datasets were generated or
analysed during the current study.

\end{document}